\documentclass[english]{article}
\usepackage[T2A]{fontenc}
\usepackage[cp1251]{inputenc}
\usepackage[english]{babel}
\usepackage{amsmath, amsfonts, amsthm, amssymb, amscd}
\usepackage[dvips]{graphicx}

\usepackage{latexsym}

\DeclareMathOperator{\link}{link} 
 
 \DeclareMathOperator{\pt}{pt}

\newcommand{\zk}{\mathcal{Z}_K}
\newcommand{\zkr}{_\Ro\mathcal{Z}_K}
\newcommand{\zl}{\mathbb{Z}^l}
\newcommand{\Zo}{\mathbb{Z}}
\newcommand{\Ro}{\mathbb{R}}
\newcommand{\Co}{\mathbb{C}}
\newcommand{\Zt}{\mathbb{Z}_2}
\newcommand{\rr}{r_{\Ro}}
\newcommand{\sr}{s_{\Ro}}
\newcommand{\Ur}{{_{\Ro}U}}

\newcounter{stmcounter}
\newcounter{thcounter}

\newtheorem*{conj}{Conjecture}
\newtheorem*{cor}{Corollary}

\newtheorem{theorem}[thcounter]{Theorem}
\newtheorem{prop}[stmcounter]{Proposition}
\newtheorem{lemma}[stmcounter]{Lemma}
\newtheorem*{defin}{Definition}
\newtheorem*{probl}{Problem}

\begin{document}

\author{Anton Ayzenberg}
\title{The problem of Buchstaber number and its combinatorial aspects}
\date{}

\maketitle

\begin{abstract}
For any simplicial complex on $m$ vertices a moment-angle complex
$\zk$ embedded in $\Co^m$ can be defined. There is a canonical
action of a group $T^m$ on $\zk$, but this action fails to be
free. The Buchstaber number is the maximal integer $s(K)$ for
which there exists a subtorus of rank $s(K)$ acting freely on
$\zk$. The similar definition can be given for real Buchstaber
number. We study these invariants using certain sequences of
simplicial complexes called universal complexes. Some general
properties of Buchstaber numbers follow from combinatorial
properties of universal complexes. In particular, we investigate
the additivity of Buchstaber invariant.
\end{abstract}

\section{Introduction}

We recall first that an (abstract) simplicial complex is a finite
set $M$ with a system of its subsets $K\subseteq 2^M$ such that:

1) If $\sigma\in K$ and $\tau\subset \sigma$, then $\tau \in K$;

2) Any singleton $\{v\}$ lies in $K$.

Elements of $M$ are called vertices. Sets from $K$ are called
simplices. We will use a notation $V(K)$ for the underlying set of
vertices $M$. The number $\dim\sigma = |\sigma| - 1$ is called the
dimension of a simplex, though we use both dimension and
cardinality. The dimension of the complex is, by definition, the
maximal dimension of its simplices.

There is a method for constructing topological spaces from another
spaces using combinatorial structure of a given simplicial
complex.

\begin{defin}
Let $K$ be a simplicial complex on $m$ vertices and $(X,A)$ be a
pair of topological spaces, $A\subseteq X$. For any simplex
$\sigma\in K$ we define the subset $(X,A)^\sigma\subset X^m$,
$(X,A)^\sigma = \{(x_1,\ldots,x_m)\in X^m, x_i\in A$, if
$i\notin\sigma\}$. Then the $K$-power of the pair $(X,A)$ is a
topological space defined as
$$
(X, A)^K = \bigcup\limits_{\sigma\in K}(X,A)^\sigma \subseteq X^m.
$$
\end{defin}

The two cases we are especially interested in are as follows:

\begin{defin}
A space $\zk = (D^2, S^1)^K$ is called a moment-angle complex of
the simplicial complex $K$. There $D^2$ denotes 2-dimensional disk
and $S^1$ --- its boundary circle.

A space $\zkr = (I, S^0)^K$ is called a real moment-angle complex
of the simplicial complex $K$. There $I$ denotes the closed
interval and $S^0$ --- its boundary.
\end{defin}

The homotopy types of spaces $\zk$ and $\zkr$ were originally
introduced by Davis and Januszkievicz in their pioneer work
\cite{DJ}. But the interpretation of these spaces as $K$-powers is
due to Buchstaber and Panov \cite{BP}.

We suppose that a pair $(D^2, S^1)$ is represented by the unitary
disk and its boundary circle in $\mathbb{C}$. So the complex $\zk$
is considered to be embedded in $\mathbb{C}^m$. Similarly the
interval $I$ can be represented as a subset $[-1,1]$ of a real
line so $\zkr$ is considered as a subspace of $\Ro^m$.

The coordinatewise action of a torus $T^m$ on $\mathbb{C}^m$
preserves $\zk$ as a subset. Thus an $m$-torus acts on $\zk$ also.
Such an action have stabilizers. To ensure their existence
consider the coordinate subgroup of a torus $T^\sigma\subset T^m$
for any simplex $\sigma\in K$. It preserves points $(x_1,\ldots,
x_m)\in \mathbb{C}^m$, where $x_i = 0$ for $i\in \sigma$ and $x_i
\neq 0$ otherwise. These points lie in $\zk\subset \mathbb{C}^m$.
So the subgroups $T^\sigma$ are stabilizers. It also can be shown
that there are no other stabilizers of this particular action. So
we see that an action of a torus is not free.

In a similar way the coordinatewise action of $\Zt^m$ on $\Ro^m$
(by multiplying each coordinate by $\pm1$) preserves the subset
$\zkr$ and thus can be restricted to this subset. The same
reasoning as in the complex case shows that stabilizers are given
by subgroups $\Zt^\sigma\subset \Zt^m$ for $\sigma\in K$.

Now we give the basic definition.

\begin{defin}
A (complex) Buchstaber number $s(K)$ is a maximal dimension of
toric subgroups in $T^m$ acting freely on $\zk$.

A real Buchstaber number $\sr(K)$ is a maximal rank of subgroups
in $\Zt^m$ acting freely on $\zkr$.
\end{defin}

So, informally speaking, the Buchstaber number shows a degree in
which the standard action of a torus (or a subgroup) fails to be
free. Now we formulate a problem posed by V.M.Buchstaber.

\begin{probl}
Find a combinatorial description of $s(K)$.
\end{probl}

The real Buchstaber number was first introduced by Yukiko Fukukawa
and Mikiya Masuda in \cite{FM}.

The aim of this article is to investigate Buchstaber number $s(K)$
of simplicial complex $K$ by combinatorial means. More
combinatorial definition of this invariant will be discussed. We
consider the sequence of universal simplicial complexes $\{U_l\},
l=1,2,\ldots$ introduced in \cite{DJ} and seek the minimal number
$r$ for which there exists a nondegenerate simplicial map from $K$
to $U_r$. Then such a number is connected with the Buchstaber
number by the simple formula $r = m - s(K)$, where $m$ is the
number of vertices of $K$. We can take any other sequence of
simplicial complexes, and it will give another invariant of $K$ in
the same way. For example, the sequence of simplices gives a
chromatic number. Using these arguments we prove a few estimations
of the Buchstaber number. Some results repeat the results of
Erokhovets \cite{Er} concerning the Buchstaber number of simple
polytopes. We will prove them using another approach.

An exact formula $$\rr(\Gamma) = r(\Gamma) =
\lceil\log_2(\gamma(\Gamma)+1)\rceil$$ connecting the Buchstaber
number and the chromatic number of $1$-dimensional simplicial
complexes is proved in the work. We also discuss additive
properties of Buchstaber number. The conjecture was that
$\rr(K\ast N) = \rr(K)+\rr(N)$ and $r(K\ast N) = r(K)+r(N)$ for
any complexes $K$ and $N$. There are many examples when this
formula holds true, but we provide a counterexample to this
conjecture.

I am grateful to Victor Buchstaber for his tasks and advices. Also
I wish to thank G\`{e}ry Debongnie for useful discussions and
comments and Nicolai Erokhovets for the attention to my work.

\section{Basic constructions}

We need a key to investigate those toric subgroups which act
freely on $\zk$. Obviously, a subgroup acting freely on $\zk$ is a
subgroup which intersects any stabilizer only in the unit. This
can be used to give an equivalent definition of Buchstaber number.
But for further considerations we need one more definition.

\begin{defin}
Consider $l$-dimensional coordinate integral lattice $\zl$. A set
of vectors $v_1, \ldots, v_k \in \zl$ is called unimodular if the
map $p: \mathbb{Z}^k\rightarrow \zl$, $p(e_i) = v_i$ is an
isomorphism to the direct summand of $\zl$. In other words, a set
$v_1, \ldots, v_k$ is a part of some basis of the lattice $\zl$.

Consider $l$-dimensional coordinate vector space $\Zt^l$. A set of
vectors $v_1, \ldots, v_k \in \Zt^l$ is called unimodular if it is
linearly independent set.

A characteristic map (real characteristic map) of a simplicial
complex $K$ to the lattice $\zl$ (resp. to $\Zt^l$) is such a map
$\Lambda: V(K)\rightarrow \zl$ (resp. $\Lambda: V(K)\rightarrow
\Zt^l$) that for any simplex $\sigma\in K$ the set of vectors
$\{\Lambda(i), i\in\sigma\}$ is unimodular.
\end{defin}

It is clear that a (real) characteristic map of $K$ to $\zl$
(resp. to $\Zt^l$) may not exist for any number $l$. For example,
if $l$ is less than or equal to $\dim(K)$, then obviously no such
map can be constructed. But for $l=m$ such a characteristic map
exists. Take, for example, $\Lambda(i) = e_i$, where $\{e_i\}$ is
a basis of a lattice (or a vector space over $\Zt$ in the real
case).

The statement below can be found in \cite{BP} in the case of
$m-\dim(K)$-dimensional subgroups. The general case is similar to
the one discussed in \cite{BP}.

\begin{prop}
There is one-to-one correspondence between $r$-dimensional toric
subgroups of $T^m$ acting freely on $\zk$ and characteristic maps
from $K$ to $\mathbb{Z}^{m-r}$, where $m$ is the number of
vertices of $K$.

There is one-to-one correspondence between subgroups of $\Zt^m$ of
rank $r$ acting freely on $\zkr$ and real characteristic maps from
$K$ to $\Zt^{m-r}$.
\end{prop}

We have the trivial corollary.

\begin{cor} \label{buch}
Let $r(K)$ be the minimal number $l$ such that there exists a
characteristic map from $K$ to $\zl$ and $\rr(K)$ be the minimal
number $l$ for which there exists a real characteristic map from
$K$ to $\Zt^l$. Then $s(K) = m - r(K)$ and $\sr(K) = m - \rr(K)$.
\end{cor}

Further on we use the numbers $r$ and $\rr$ instead of $s$ and
$\sr$ bearing in mind that the original Buchstaber numbers can be
calculated from the former ones by the simple formula from the
corollary.

Characteristic maps can be treated more geometrically. For this we
use a notion of a universal complex introduced by Davis and
Januszkievicz in \cite{DJ}. First we construct a simplicial
complex $U_l$. Let the primitive nonzero vectors $v\in \zl$ (those
are the vectors whose coordinates are relatively prime) be
vertices of $U_l$. All unimodular sets define simplices of $U_l$.
The dimension of the complex $U_l$ is $l-1$ since the maximal
cardinality of a unimodular set in $\zl$ is $l$. Note that the
maximal simplices in $U_l$ correspond to the bases of the lattice
$\zl$. Using this space $U_l$ a characteristic map from $K$ to
$\zl$ can be treated as a nondegenerate simplicial map from $K$ to
$U_l$. In a real case we also define a simplicial complex $\Ur_l$
whose vertices are nonzero vectors of $\Zt^l$ and whose simplices
are unimodular (in other words linearly independent) sets of
vectors. The maximal simplices are bases of the space $\Zt^l$. In
terms of this space a real characteristic map from complex $K$ to
$\Zt^l$ is just a nondegenerate simplicial map from $K$ to
$\Ur_l$.

Note that in the complex case there exists an antipodal map on the
complex $U_l$ converting $v\in \zl$ to $-v\in\zl$. In \cite{DJ}
characteristic maps are considered modulo this antipodal map, and
the definition of the universal complex slightly differs from the
one given above. But this difference makes no sense in what we are
going to do.

The term "universal" originally came from the observation by Davis
and Januszkievicz that there is a certain object in the category
of quasitoric spaces (small covers in the real case) over
simplicial complexes which is similar to universal spaces for
principal bundles. This universal object have $U_l$ ($\Ur_l$ in
the case of small covers) as its base.

We give another reformulation in the real case. One may consider
nonzero points of $\Zt^l$ as points of $l-1$-dimensional
projective space $\Zt P^{l-1}$ over the field $\Zt$. Then a set of
vectors is unimodular iff corresponding points of $\Zt P^{l-1}$
are affinely independent. Thus simplices of $\Ur_l$ can be thought
of as "simplices" in finite projective geometry $\Zt P^{l-1}$. A
nondegenerate simplicial map from complex $K$ to $\Ur_l$ can be
considered as an "immersion" of a complex $K$ into a finite
geometry $\Zt P^{l-1}$. Such a point of view makes transparent the
connection between problem posed by Buchstaber and a classical
theory of immersions.

\section{Invariants defined by sequences}
In this section we provide a proof for some results on the
Buchstaber number using a general technique. A different proof of
some of these facts can be found at \cite{Er}, \cite{Izm}.

Let $\{L_i, i = 1,2,\ldots\}$ be a sequence of simplicial
complexes such that $L_i$ can be mapped nondegenerately to $L_j$
for $i<j$. Such a sequence will be called an increasing sequence.
Consider an arbitrary simplicial complex $K$. For any increasing
sequence $\{L_i, i=1,2,\ldots\}$ we define a number $L(K)$ as the
smallest number $l$ for which there exists a nondegenerate map
from $K$ to $L_l$. If there is no such number, we set $L(K) =
\infty$. We say that an increasing sequence $\{L_i\}$ defines an
invariant $L$ of a simplicial complex.

\textit{Example.} We have already seen in the previous section
that the increasing sequences of universal complexes $\{U_i,
i=1,2,\ldots\}$ and $\{\Ur_i, i=1,2,\ldots\}$ define the
invariants $r$ and $\rr$ respectively. Consider the sequence
$\{L_i = \Delta_{i-1}, i=1,2,\ldots\}$ of simplices. This sequence
defines a chromatic number $\gamma$. Indeed, a nondegenerate map
from $K$ to $\Delta_{i-1}$ corresponds to the coloring of vertices
by $i$ paints in which no two adjacent vertices are colored by the
same paint. Another example is given by the sequence
$\{\Delta_\infty^{(i-1)}, i=1,2,\ldots\}$ consisting of complexes
$\Delta_\infty^{(i-1)}$ with infinitely (countably) many vertices
and maximal simplices defined by all $i$-subsets of vertices. This
sequence defines an invariant equal to the dimension plus 1.

\begin{prop}\label{series}
Let $\{L^1_i\}$ and $\{L_i^2\}$ be increasing sequences and $L^1$
and $L^2$ be the invariants defined by these sequences. Suppose
that for each $i$ there exists a nondegenerate map $g_i$ from
$L_i^1$ to $L_i^2$. Then for each simplicial complex $K$ there
holds $L^1(K)\geqslant L^2(K)$.
\end{prop}

\begin{proof}
For any nondegenerate map $f$ from $K$ to $L_i^1$ the composition
$g_i\circ f$ is a nondegenerate map from $K$ to $L_i^2$.
Substituting $i=L^1(K)$ leads to the required estimation.
\end{proof}

The estimation for Buchstaber number and chromatic invariant (see
\cite{Izm}), and the estimation for real and complex Buchstaber
numbers (\cite{FM}) can be deduced from proposition~\ref{series}.

\begin{prop} \label{buchchrom}
For any complex $K$ there holds $$\dim(K)+1\leqslant
\rr(K)\leqslant r(K)\leqslant\gamma(K).$$
\end{prop}

\begin{proof}
We prove that for each $i$ there exist nondegenerate maps
$$\Delta_{i-1}\to U_i \to \Ur_i\to\Delta_\infty^{(i-1)}.$$ For the first
map take for example an inclusion of any maximal simplex into
$U_i$. The existence of the last map is also obvious since the
dimension of $\Ur_i$ is $i-1$. We have to construct a map in the
middle. Recall that vertices of $U_i$ are nonzero primitive
vectors of $\Zo^i$ and vertices of $\Ur_i$ are nonzero vectors of
$\Zt^i$. We define a map from $U_i$ to $\Ur_i$ on vertices by
reduction modulo two. It is clear that maximal simplices in $U_i$
(which are bases of the integral lattice) go to maximal simplices
of $\Ur_i$ (which are bases of the corresponding vector space over
$\Zt$). So the constructed map is simplicial and nondegenerate.

The assertion of the proposition now follows from proposition
\ref{series}.
\end{proof}

Next statement in the case of Buchstaber number was originally
observed by Nicolai Erokhovets in \cite{Er}. The proof of this
proposition is similar to that of proposition \ref{series} and
thus omitted.

\begin{prop}\label{estgen}
Let $\mathcal{L}$ be an invariant defined by some increasing
sequence $\{L_i\}$. Let $K$ and $N$ be simplicial complexes and
there is a nondegenerate map from $K$ to $N$. Then
$\mathcal{L}(K)\leqslant\mathcal{L}(N)$.
\end{prop}

This observation happens to be very useful. For example there is a
nondegenerate map from any given complex $K$ to
$\Delta_{\gamma-1}$ where $\gamma = \gamma(K)$, which follows from
the description of the chromatic number given in example. Note
that this nondegenerate map can be viewed as a map from $K$ to
$\Delta_{\gamma-1}^{(\dim(K))}$. Thus we have the estimations
$r(K)\leqslant r(\Delta_{\gamma-1}^{(\dim(K))})$ and
$\rr(K)\leqslant \rr(\Delta_{\gamma-1}^{(\dim(K))})$ by the
previous statement. These estimations involve the evaluation of
the Buchstaber number for the skeletons of simplices. But even
this particular case is a challenge. For an extensive information
on the real Buchstaber number of skeletons of simplices see
\cite{FM}.

We use the proposition \ref{estgen} in a different way by changing
sequences under consideration. By definition, there exists a
nondegenerate map from a complex $K$ to $U_{r(K)}$ and a
nondegenerate map from $K$ to $\Ur_{\rr(K)}$. So we have
\begin{equation}\label{estchrom}
\gamma(K)\leqslant\gamma(U_{r(K)}),
\end{equation}
\begin{equation}
\gamma(K)\leqslant\gamma(\Ur_{\rr(K)})
\end{equation}
by proposition \ref{estgen}. To obtain more transparent formulae
chromatic numbers of universal complexes should be found. This can
be done directly, but we prefer to formulate a few lemmas of an
independent interest.

We call two simplicial complexes $K$ and $N$ \textit{equivalent}
if there exist nondegenerate maps from $K$ to $N$ and from $N$ to
$K$. In these terms the invariant defined by any sequence of
complexes is invariant under this equivalence relation.

\begin{lemma}\label{equiv2}
For any $l\geqslant1$ the complex $U_l^{(2)}$ is equivalent to the
complex $\Ur_l^{(2)}$.
\end{lemma}

\begin{proof}
A nondegenerate map from $U_l$ to $\Ur_l$ was already constructed
in the proof of \ref{buchchrom}. The restriction of this map to
the $2$-skeleton of $U_l$ gives one of the maps required in the
equivalence relation.

Now we construct a nondegenerate map $q$ from $\Ur_l^{(2)}$ to
$U_l^{(2)}$. Let $v\in\mathbb{Z}_2^l$, $v =
(\delta_1,\ldots,\delta_l)$, where $\delta_i = 0$ or $1\mod 2$. We
set $q(v) = (\delta_1,\ldots,\delta_l)$ --- the same row-vector of
zeros and units considered as an integral vector. We should now
check that any $2$-dimensional simplex goes to $2$-dimensional
simplex under $q$. Any $2$-dimensional simplex in $U_l(2)$ is
given by three linearly independent vectors. Let us write down
their coordinates into $3\times l$-matrix

$$A =
\begin{pmatrix}
\varepsilon_{1,1}&\ldots&\varepsilon_{1,i_1}&\ldots&\varepsilon_{1,i_2}&\ldots&\varepsilon_{1,i_3}&\ldots&\varepsilon_{1,l}\\
\varepsilon_{2,1}&\ldots&\varepsilon_{2,i_1}&\ldots&\varepsilon_{2,i_2}&\ldots&\varepsilon_{2,i_3}&\ldots&\varepsilon_{2,l}\\
\varepsilon_{3,1}&\ldots&\varepsilon_{3,i_1}&\ldots&\varepsilon_{3,i_2}&\ldots&\varepsilon_{3,i_3}&\ldots&\varepsilon_{3,l}
\end{pmatrix}
$$

This matrix has a minor which is nonzero over $\mathbb{Z}_2$. Thus
the corresponding $3\times3$ matrix $M$ considered as a matrix
over $\mathbb{Z}$ has an odd determinant. We use a simple fact
that any $3\times 3$ matrix $M$ over $\mathbb{Z}$ filled with
zeros and units satisfies the relation $|\det M|<3$. It now
follows that $\det M = \pm1$. Therefore the matrix $A$ over
$\mathbb{Z}$ has the minor $M$ equal to $\pm1$. This means that
its rows form a part of some basis of a lattice $\zl$. So
$q(\sigma)\in U_l$ for any $2$-dimensional simplex $\sigma\in
U_l(2)$.
\end{proof}

\begin{cor} \label{k2l}
The complexes $U_l^{(1)}$, $\Ur_l^{(1)}$ and complete graph
$K_{2^l-1}$ on $2^l-1$ vertices are equivalent.
\end{cor}

\begin{proof}
By the previous lemma we have $U_l^{(1)}\sim \Ur_l^{(1)}$. But
$\Ur_l^{(1)} = K_{2^l-1}$ because any two different nonzero
vectors from $\mathbb{Z}_2^l$ are linearly independent, thus form
a $1$-simplex in $\Ur_l$.
\end{proof}

Now we can easily find a chromatic number of a universal complex.

\begin{prop}
For any $l\geqslant1$ there holds
$$\gamma(U_l) = \gamma(\Ur_l) = 2^l-1.$$
\end{prop}

\begin{proof}
Chromatic number depends only on the $1$-skeleton of a complex.
Moreover, it is defined by some sequence, therefore it is an
invariant of an equivalence relation. Finally, using previous
corollary we obtain $\gamma(U_l) = \gamma(\Ur_l) =
\gamma(U_l^{(1)}) = \gamma(\Ur_l^{(1)}) = \gamma(K_{2^l-1}) =
2^l-1$.
\end{proof}

Now we may rewrite estimations \ref{estchrom} as
\begin{equation}
\gamma(K)\leqslant 2^{r(K)}-1,
\end{equation}
\begin{equation}
\gamma(K)\leqslant 2^{\rr(K)}-1
\end{equation}
and get the estimation for numbers $r$ and $\rr$
\begin{equation}\label{logchrom}
r(K)\geqslant \lceil\log_2(\gamma(K)+1)\rceil,
\end{equation}
\begin{equation}
\rr(K)\geqslant \lceil\log_2(\gamma(K)+1)\rceil.
\end{equation}
There $\lceil a\rceil$ is the least integer which is greater than
or equal to $a$.

We now prove that this estimation attains for simple graphs
(1-dimensional simplicial complexes).

\begin{theorem}
For any simple graph $\Gamma$ we have a formula
\begin{equation} \label{theorgraph}
\rr(\Gamma) = r(\Gamma) = \lceil\log_2(\gamma(\Gamma)+1)\rceil,
\end{equation}
\end{theorem}

\begin{proof}
Any coloring of a graph $\Gamma$ by $a$ colors defines a
nondegenerate map from $\Gamma$ to $K_a$. Thus there exists a
nondegenerate map from $\Gamma$ to $K_{\gamma(\Gamma)}$. Denote
$\lceil\log_2(\gamma(\Gamma)+1)\rceil$ by $p$. Then
$2^p-1\geqslant \gamma(\Gamma)$ and therefore there exists a
nondegenerate map from graph $\Gamma$ to $K_{2^p-1}$. But
$K_{2^p-1}\sim U_p^{(1)}$ by corollary of lemma \ref{equiv2}. So
there exists nondegenerate map from $\Gamma$ to $U_p$ where $p =
\lceil\log_2(\gamma(\Gamma)+1)\rceil$. This observation shows that
$r(\Gamma) \leqslant \lceil\log_2(\gamma(\Gamma)+1)\rceil$.
Combining it with the estimation $r(\Gamma)\geqslant
\lceil\log_2(\gamma(\Gamma)+1)\rceil$ given by (\ref{logchrom}) we
get the asserted formula. Argumentation is the same in the case of
$\rr$.
\end{proof}

Lemma \ref{equiv2} implies that real and complex Buchstaber
numbers of $2$-dimensional complexes coincide. Indeed, from the
existence of a nondegenerate map from $K$ to $\Ur_l^{(2)}$ follows
the existence of a nondegenerate map from $K$ to $U_l^{(2)}$ and
vice versa. But in the dimension $3$ real and complex Buchstaber
numbers may not coincide as the following proposition shows.

\begin{prop}
There is no nondegenerate map from $\Ur_4$ to $U_4$.
\end{prop}

This gives $\rr(\Ur_4) = 4 \neq r(\Ur_4)$. The proof consists in
the examination of a big number of possibilities, that can be made
partially by a computer search. We do not give the proof here.

\section{Combinatorial properties of universal complexes}

In this section we treat different properties of a Buchstaber
number as consequences of some geometrical and combinatorial
properties of universal complexes.

\begin{prop} \label{linkUl}
1) Let $\sigma\in U_l$, $|\sigma| = k$. Then $\link_{U_l}\sigma
\sim U_{l-k}$.

2) There exists a nondegenerate map from $U_l\ast U_k$ to
$U_{l+k}$.

Similar statements hold for the real universal complexes.
\end{prop}

\begin{proof}
We provide a proof only in the complex case. The real case can be
proved by the same reasoning.

1) Any automorphism of $\zl$ determines an automorphism of the
simplicial complex $U_l$. Therefore without loss of generality we
can prove the statement only for the simplex $\sigma$ whose
vertices are the first $k$ vectors of the standard basis.

At first, we construct a nondegenerate map $p$ from $U_{l-k}$ to
$\link_{U_l}\sigma$. For $v =
\linebreak(v_1,\ldots,v_{l-k})\in~U_{l-k}$ we set $p(v) =
(0,\ldots,0,v_1,\ldots,v_{l-k})\in U_l$. This map is simplicial
and nondegenerate. Moreover, its image lies in
$\link_{U_l}\sigma$.

Now we construct a nondegenerate map $q:
\link_{U_l}\sigma\rightarrow U_{l-k}$. Let $D$ be a subgroup of a
lattice generated by vertices of $\sigma$. Since $D$ is a direct
summand, there exists a surjective quotient map
$r:\zl\rightarrow\zl/D \cong \mathbb{Z}^{l-k}$. Let us show that
this map induces a nondegenerate map from $\link_{U_l}\sigma$ to
$U_{l-k}$. Let $\tau\in\link_{U_l}\sigma$ or, in other words,
$\tau\sqcup\sigma$ is a part of some basis of a lattice $\zl$.
Suppose $\tau$ is the maximal simplex of $\link_{U_l}\sigma$,
$\tau\sqcup\sigma$ being a basis. In the converse case we complete
$\tau$ to the maximal simplex. Then the vectors
$r(\tau\sqcup\sigma) = r(\tau)\sqcup r(\sigma) = r(\tau)\sqcup
\{0\}$ generate $\mathbb{Z}^{l-k}$. Since
$|r(\tau)|\leqslant|\tau| = l-k$, the set $r(\tau)$ is a basis of
$\mathbb{Z}^{l-k}$ and, therefore, it is a simplex of $U_{l-k}$
with $|r(\tau)| = |\tau|$. This concludes the proof of the first
statement.

2) Consider a decomposition $\mathbb{Z}^{l+k} =
\mathbb{Z}^l\oplus\mathbb{Z}^k$. Let $p: U_l\ast U_k\rightarrow
U_{l+k}$ be the map defined by the rule: $U_l$ maps on the first
summand and $U_k$ maps on the second summand of the decomposition
in the obvious way. It can be proved directly that the map $p$ is
simplicial and nondegenerate.
\end{proof}

This yields an estimation for $r$-numbers of join of complexes
(see \cite{Er} for more algebraic explanation).

\begin{prop}
$$r(K\ast N) \leqslant r(K)+r(N),$$
$$\rr(K\ast N) \leqslant \rr(K)+\rr(N),$$
$$r(K\ast N) \geqslant r(K) + \dim(N) + 1,$$
$$\rr(K\ast N) \geqslant \rr(K) + \dim(N) + 1.$$
\end{prop}

\begin{proof}
The case of $r$ is considered only.

There exists a nondegenerate map $p:K\rightarrow U_{l_1}$ and a
nondegenerate map $q: N\rightarrow U_{l_2}$ where $l_1 = r(K)$,
$l_2 = r(N)$. Therefore we have a nondegenerate map $p\ast q:
K\ast N\rightarrow U_{l_1}\ast U_{l_2}$. Composing it with a
nondegenerate map from $U_{l_1}\ast U_{l_2}$ to $U_{l_1+l_2}$ we
obtain a nondegenerate map from $K\ast N$ to $U_{l_1+l_2}$. This
completes the proof of the first formula.

We turn to the proof of the third formula. Denote $r(K\ast N)$ by
$l$. Let $\sigma$ be a simplex of $N$ of maximal dimension. Then
the subcomplex $\link_{K\ast N}\sigma = \link_N\sigma\ast K = K$
can be mapped nondegenerately to $\link_{U_l}\Delta$, where
$|\Delta| = \dim(N) + 1$. But $\link_{U_l}\Delta\sim
U_{l-\dim(N)-1}$, therefore, $\link_{K\ast N}\sigma$ can be mapped
nondegenerarely into $U_{l-\dim(N) - 1}$. This implies
$r(K)\leqslant l - \dim(N) - 1 = r(K\ast N) - \dim(N) - 1$.
\end{proof}

We call the complex $K$ \textit{optimal} if $r(K) = \dim(K)+1$. It
now follows from the previous statement that $r(K\ast N) =
r(K)+r(N)$ if one of the complexes $K$ and $N$ is optimal.

Optimal complexes play an important role in toric topology. If we
are given a simple polytope $P$ and the boundary of its dual
$\partial P^\ast$ is an optimal simplicial complex, then there
exists a quasitoric manifold over the polytope $P$. More details
on quasitoric manifolds can be found in \cite{DJ}, \cite{BP}.

\begin{conj} \label{conj}
For any simplicial complexes $K$ and $N$
$$r(K\ast N) = r(K)+r(N)$$
$$\rr(K\ast N) = \rr(K)+\rr(N)$$
\end{conj}

This conjecture fails in general, the counterexample will be given
below. We first show that this formula holds true for complete
graphs even if none of them is optimal.

\begin{prop}\label{prop}
$$\rr(K_p\ast K_q) = \rr(K_p)+ \rr(K_q),$$
$$r(K_p\ast K_q) = r(K_p)+ r(K_q).$$
\end{prop}

Let us show that the first formula implies the second one. Indeed,
suppose that $r(K_p\ast K_q) < r(K_p)+ r(K_q)$. Then
$$\rr(K_p\ast K_q)\leqslant r(K_p\ast K_q) < r(K_p)+ r(K_q) = \rr(K_p)+ \rr(K_q).$$
So we need to prove only the real case.

Consider a nondegenerate map $f\colon K\ast N\to \Ur_l$. The
images of vertices of a complex $K$ are denoted by $x_i,
i=1,\ldots,m_K$ and the images of vertices of $N$ are denoted by
$y_j, j=1,\ldots,m_N$. All these images are considered as nonzero
vectors of $\Zt^l$. A condition of nondegeneracy now implies that
$\{x_i\}_{i\in \sigma}\sqcup \{y_j\}_{j\in \tau}$ is a linearly
independent set of vectors for $\sigma\in K$ and $\tau \in N$.
Equivalently: $\{x_i\}_{i\in\sigma}$ are linearly independent,
$\{y_j\}_{j\in\tau}$ are linearly independent and sets of sums $A
= \{\sum\limits_{i\in\sigma}x_i$ for $\sigma\in K,
\sigma\neq\varnothing\}$ and $B = \{\sum\limits_{j\in\tau}y_j$ for
$\tau\in N, \tau\neq\varnothing\}$ do not intersect. In more
conceptual terms: a nondegenerate map from $K$ to $U_l$ defines an
arrangement of subspaces, each subspace being a span of
$\{x_i\}_{i\in \sigma}$ for simplices $\sigma\in K$. Then a
nondegenerate map from $K\ast N$ to $\Ur_l$ is defined iff both
nondegenerate maps from $K$ and $N$ to $\Ur_l$ are defined and
corresponding arrangements intersect each other only in zero.

For the particular case of complete graphs $K = K_p, N = K_q$ the
described condition of nondegeneracy has the form: all $x_i$ are
different for $i=1,\ldots, p$, all $y_j$ are different for
$j=1,\ldots, q$, and sets $A =
\{x_i\}\cup\{x_\alpha+x_\beta\}_{\alpha\neq\beta}$ and $B =
\{y_j\}\cup\{x_\gamma+x_\delta\}_{\gamma\neq\delta}$ are disjoint.
In this case we call the sets $\{x_i\}_{i=1,\ldots,p}$ and
$\{y_j\}_{j=1,\ldots,q}$ a good pair. Next lemma shows how new
good pairs can be constructed from the given one.

\begin{lemma} \label{constrlem}
If $\{x_1,\ldots,x_\alpha,\ldots,x_p\}$ and
$\{y_1,\ldots,y_\gamma,\ldots,y_q\}$ is a good pair then
$\{x_1+y_\gamma,\ldots,
x_{\alpha-1}+y_\gamma,x_\alpha,x_{\alpha+1}+y_\gamma,\ldots,
x_p+y_\gamma\}$ and $\{y_1+x_\alpha,\ldots,y_{\gamma-1}+x_\alpha,
y_\gamma,y_{\gamma+1}+x_\alpha,\ldots,y_q+x_\alpha\}$ is also a
good pair. The pair
$\{x_1+x_\alpha,\ldots,x_\alpha,\ldots,x_p+x_\alpha\}$ and
$\{y_1,\ldots,y_\gamma,\ldots,y_q\}$ is good.
\end{lemma}

The proof is straightforward. Eventually such transformations work
only for complete graphs. Note that first transformation changes
both sets of a pair and second transformation changes only one of
the sets. We denote first transformation by $t_{\alpha\gamma}$ and
second transformation --- by $t^1_\alpha$ (where the index $1$
denotes that it is applied to the first set). Both
$t_{\alpha\gamma}$ and $t^1_\alpha$ are idempotents. Next lemma
shows that a certain composition of such elementary
transformations looks quite simple (if we look only on one set).

\begin{lemma}
Applying the transformation $t_{\alpha\delta}\circ
t_{\beta\gamma}\circ t_{\beta\delta}\circ t_{\alpha\gamma}$ to a
good pair $(\{x_i\}_{i=1,\ldots,p},\linebreak
\{y_j\}_{j=1,\ldots,q})$ for $\alpha\neq\beta, \gamma\neq\delta$
gives a new good pair, whose second set is
$\{y_1,\ldots,y_\gamma+x_\beta,\ldots,y_\delta+x_\beta,\ldots,y_q\}$.
This set differs from the original set in two entries by an
element $x_\beta$.
\end{lemma}

The proof consists in consecutive performing of transformations.
The aim of this lemma is following: given a good pair $(S,T)$ one
can construct another good pair $(S',T')$ whose second set $T'$ is
obtained from $T$ by adding the same vector from $S$ to two
different elements of $T$.

\begin{lemma}\label{hyperlem}
Let $q,p>1$ and $q$ be an even number. If there exists a good pair
$S=\{x_i\}_{i=1,\ldots,p}$ and $T=\{y_j\}_{j=1,\ldots,q}$ in
$\Zt^l$ then there exists a good pair
$\{\tilde{x}_i\}_{i=1,\ldots,p}$ and
$\{\tilde{y}_j\}_{j=1,\ldots,q}$ in $\Zt^l$ in which one of the
sets lies in a hyperplane.
\end{lemma}

\begin{proof}
Fix an arbitrary hyperplane. We take for instance a hyperplane
$\Pi$ given by $v_1=0$. We now try to transform a given pair to
"push" the set $T$ into the hyperplane $\Pi$. Suppose there are
$k$ elements of $T$ which do not lie in $\Pi$ (thus having $1$ as
their first coordinate). Without loss of generality we assume $k$
is even. Indeed, let $k$ be an odd number. Then take one of the
vectors which do not lie in $\Pi$, say $y_\gamma$ and apply a
transformation $t_\gamma^2$ to $T$. From now on vectors outside
the hyperplane are: a vector $y_\gamma$ and those vectors which
lied in $\Pi$ before the transformation. So there is an even
number of vectors outside $\Pi$ after transformation.

Now we use an induction on an even number $k$. If $k=0$ then lemma
is proved. Suppose $k>0$. Take any vector $x_\beta$ from $S$ which
do not lie in $\Pi$. If there is no such vector, the lemma is
proved since $S$ lies in $\Pi$. Take any other vector $x_\alpha$
from the set $S$ (it exists since $p>1$). Finally, take two
vectors $y_\gamma$ and $y_\delta$ from $T$ which do not lie in
$\Pi$ and apply a transformation $t_{\alpha\delta}\circ
t_{\beta\gamma}\circ t_{\beta\delta}\circ t_{\alpha\gamma}$. This
adds $x_\beta$ to $y_\gamma$ and $y_\delta$ pushing them into the
hyperplane $\Pi$ and do not change other elements of $T$. A number
of elements of $T$ outside $\Pi$ has now being reduced, so we may
apply an induction hypothesis.
\end{proof}

\begin{proof}[Proof of the proposition \ref{prop}]
We are given two complete graphs $K_p$ and $K_q$. Suppose $k$ and
$n$ are maximal integers such that $2^k\leqslant p$ and
$2^n\leqslant q$. Then $\rr(K_p) = \rr(K_{2^k}) = k+1$ and
$\rr(K_q)=\rr(K_{2^n})=n+1$. Suppose that $\rr(K_p\ast
K_q)<\rr(K_p)+\rr(K_q)$. Then $$\rr(K_{2^k}\ast K_{2^n})\leqslant
\rr(K_p\ast K_q)<\rr(K_p)+\rr(K_q) = \rr(K_{2^k})+\rr(K_{2^n}).$$
So the only case under consideration is the case when both $p$ and
$q$ are powers of $2$. From now on $p=2^k$, $q=2^n$.

We claim that there is no nondegenerate map from $K_{2^k}\ast
K_{2^n}$ to $\Ur_{k+n+1}$. We prove this statement by induction on
$k+n$. Suppose there exists a nondegenerate map from $K_p\ast K_q$
to $\Ur_{k+n+1}$. This means that there is a good pair $(S,T)$
with $|S|=p, |T|=q$ in $\Zt^{k+n+1}$. There are two possibilities:

1) One of the numbers $p, q$ is $1$. Then one of the complexes
under consideration is optimal and the proposition holds true.

2) None of the numbers $p,q$ is $1$. Then by lemma \ref{hyperlem}
we can assume that $B$ lies in some hyperplane $\Pi\cong
\Zt^{k+n}$. Let $S_{\Pi} = S\cap \Pi$ and $s_{\Pi}$ be a
cardinality of $S_{\Pi}$. There are two cases:

1') $s_{\Pi}\geqslant 2^{k-1}$. Then a pair $(S_{\Pi},T)$ is a
good pair in $\Pi\cong\Zt^{k+n}$. In this case there exists a
nondegenerate map from $K_{2^{k-1}}\ast K_{2^n}$ to $\Ur_{k+n}$.
This contradicts the induction hypothesis.

2') $s_{\Pi} < 2^{k-1}$. Then there are at least $2^{k-1}+1$
elements of $S$ that do not lie in $\Pi$. Take one of them, say
$x_\alpha$ and apply a transformation $t_\alpha^1$ to the first
set. This transformation puts into $\Pi$ all vectors that were
outside $\Pi$ except $x_\alpha$. After the transformation there
are at least $2^{k-1}$ vectors of $S$ that lie in $\Pi$. Note that
the transformation did not change the set $T$. Now we are in the
conditions of the first case.

The last thing to be checked is the base of induction $k+n=0$.
Obviously there is no nondegenerate map from $\pt\ast\pt =
\Delta_1$ to $\Ur_1=\pt$.
\end{proof}

Now we construct a counterexample to the conjecture. Consider two
graphs $K$ and $N$ depicted in figures \ref{pic1} and \ref{pic2}.
One of them is a complete graph with $4$ vertices. Another one is
Gr\"{o}tzsch graph. The chromatic number of both graphs is $4$.
Therefore by (\ref{theorgraph}):
$$r(K)=r(N)=\rr(K)=\rr(N)=3.$$
We now show that $\rr(K\ast N) \leqslant 5$. To do this a
nondegenerate map from $K\ast N$ to $\Ur_5$ should be constructed.
So we need to assign a nonzero vector $x_i$ of $\Zt^5$ to each
vertex $i$ of $K$ and a vector $y_j$ to each vertex $j$ of $N$. As
was mentioned above the nondegeneracy condition for the map from
the join $K\ast N$ to $\Ur_l$ is equivalent to the following one:

\begin{figure}
\begin{center}
\includegraphics[scale=0.7]{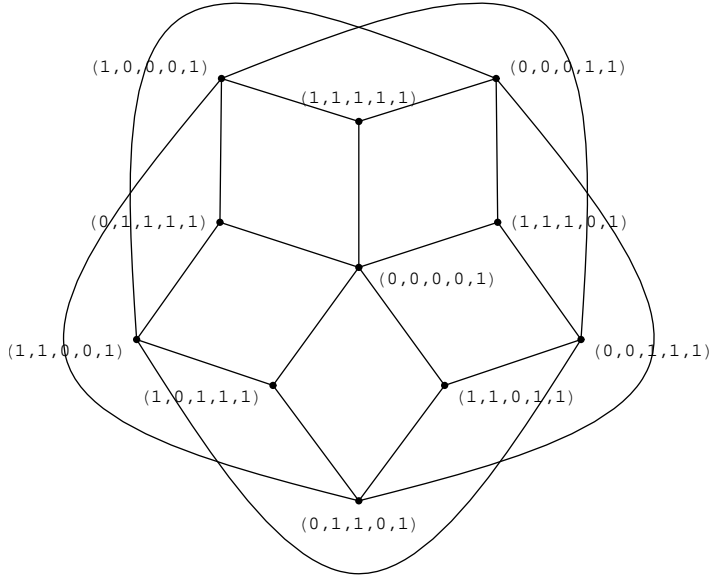}
\end{center}
\caption{Gr\"{o}tzsch graph $K$}\label{pic1}
\end{figure}

\begin{figure}
\begin{center}
\includegraphics[scale=0.7]{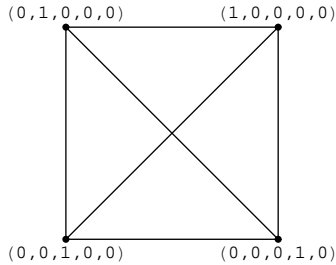}
\end{center}
\caption{Complete graph $N$}\label{pic2}
\end{figure}

\textit{Any adjoint vertices of $K$ should be assigned different
vectors. Any adjoint vertices of $N$ should be assigned different
vectors. Let $A = \{x_i\}\cup\{x_{i_1}+x_{i_2}$ for all edges
$(i_1,i_2)$ of $K\}$ and $B = \{y_j\}\cup\{y_{j_1}+y_{j_2}$ for
all edges $(j_1,j_2)$ of $N\}$. Then $A$ and $B$ should not
intersect.}

It is shown on figures \ref{pic1} and \ref{pic2} how to assign
vectors to vertices of graphs to satisfy the condition. A little
explanation is needed. All the vectors assigned to the vertices of
the first graph have $1$ as their last entry. Their pairwise sums
assigned to each edge of the first graph have three or four units
in their binary expansion as one can see from the figure. But all
the vectors assigned to the vertices of the second graph and their
pairwise sums have zero at the last position and have at most two
units in their binary expansion. Thus vectors which correspond to
vertices and edges of the first graph are disjoint from vectors
corresponding to vertices and edges of the second graph. It
constitutes the condition of nondegeneracy.

Now let us demonstrate that $r(K\ast N)\leqslant 5$. By the
definition of $r$-number we need to construct a nondegenerate map
from $K\ast N$ to $U_{5}$. To do this we assign a vector from
$\mathbb{Z}^5$ to each vertex of $K$ and $N$ in the way shown in
figures \ref{pic1} and \ref{pic2} (but in this case we consider
these vectors as integral vectors). The only thing to check is the
nondegeneracy condition: for any simplex $\sigma\in K\ast N$ the
vectors corresponding to the vertices of $\sigma$ should form
unimodular set. Suppose $\sigma$ is the maximal simplex of $K\ast
N$ and $\sigma = \tau\sqcup\rho$, where $\tau = \{v_1,v_2\}\in K$,
$\rho = \{u_1,u_2\}\in N$. Let $k(v_1), k(v_2), k(u_1), k(u_2)$ be
$5$-dimensional integral vectors corresponding to vertices
$v_1,v_2,u_1,u_2$. Consider an integral $4\times 5$-matrix $A$
whose rows are $k(v_1), k(v_2), k(u_1), k(u_2)$. To prove that
$\{k(v_1), k(v_2), k(u_1), k(u_2)\}$ is a unimodular set we will
show that the matrix $A$ has a minor equal to $\pm1$. The
reduction of $A$ modulo $2$ has a minor $M$ which is nonzero over
$\Zt$ since $\{k(v_1), k(v_2), k(u_1), k(u_2)\}\mod2$ are linearly
independent over $\Zt$ by the previous discussion. Thus $M$ is
odd. The matrix which corresponds to the minor $M$ have at least
two rows with one unit therefore $M$ equals the determinant of
some $2\times 2$-matrix made of zeros and units. It now follows
that $M=\pm1$ by the reasoning that was used in the proof of lemma
\ref{equiv2}.

These constructions confirm that both $\rr$ and $r$ are not
additive in general.

\section{Chromatic-like invariants}

In this section we continue investigating Buchstaber number using
increasing sequences of spaces. We introduce new type of
invariants.

\begin{defin}
Let $K$ be a simplicial complex. The coloring of its vertices $c:
V(K)\rightarrow [k]$ is called $q$-regular if the vertices of any
$q$-simplex are not labeled by the same color. Define
$\gamma_q(K)$ as the minimal number of paints needed for the
$q$-regular coloring of $K$.
\end{defin}

\textit{Remark.} If one consider a hypergraph whose vertices are
vertices of $K$ and hyperedges are $q$-simplices of $K$ then
classical definition of coloring for this hypergraph is the same
as $q$-regular coloring introduced above.

\textit{Example.} $\gamma_1$ is just a chromatic invariant
$\gamma$.

An invariant $\gamma_q$ can be defined via increasing sequences of
simplicial complexes. For any fixed $q$ consider the sequence
$L^{q}_i = \left(\Delta_\infty^{(q-1)}\right)^{\ast i}$ (an
exponent suggests the join of $i$ copies). Then $L^{q}_i$ defines
the invariant $\gamma_q$. Indeed, for $q$-regular coloring of $K$
by $i$ paints we can associate a map from $K$ to
$\left(\Delta_\infty^{(q-1)}\right)^{\ast i}$ which transfers
elements colored by $j$-th paint into different points of $j$-th
factor of $\left(\Delta_\infty^{(q-1)}\right)^{\ast i}$ for
$j=1,\ldots,i$. This map is simplicial and nondegenerate. On the
other hand any such map produce a $q$-regular coloring.

Note that the complex $\left(\Delta_\infty^{0}\right)^{\ast i}$ is
equivalent to $\Delta_{i-1}$, which gives a sequence of simplices
for chromatic number $\gamma = \gamma_1$.

As a consequence of proposition \ref{estgen} we may formulate

\begin{prop}
For any simplicial complex $K$ there holds
$$\gamma_q(K)\leqslant\gamma_q(U_{r(K)}),$$
$$\gamma_q(K)\leqslant\gamma_q(\Ur_{\rr(K)})$$
\end{prop}

As we will see in some cases these estimations are stronger than
\ref{logchrom}. The only thing to do is to find (or at least
estimate) numbers $\gamma_q(U_l)$ and $\gamma_q(\Ur_l)$. The
interesting result appears in real case.

\begin{theorem} \label{genchrombuch}
For any $l,q\geqslant1$ there holds
$$\gamma_q(\Ur_l)\geqslant \frac{2^l-1}{2^q-1},$$
and for $q|l$ there is an equality
$$\gamma_q(\Ur_l) = \frac{2^l-1}{2^q-1}.$$
\end{theorem}

\textit{Example.} Consider a complex $K = \Delta_{62}^{(2)}$.
Suppose that there exists a nondegenerate map from $K$ to $\Ur_6$.
Then $\gamma_2(K)\leqslant\gamma_2(\Ur_6)$. But one can see that
$\gamma_2(K) = 32$ (since no three vertices of $K$ are colored by
the same paint) and $\gamma_2(\Ur_6) = \frac{2^6-1}{2^2-1} = 21$
by the theorem \ref{genchrombuch}. This contradiction shows that
$\rr(\Delta_{63}^{(2)})\geqslant7$. Note that this cannot be
deduced from estimation~(\ref{logchrom}).

\begin{proof}[Proof of the theorem \ref{genchrombuch}]
We claim that the maximal number of vertices of $\Ur_l$ which can
be colored by the same paint is $2^q-1$. Indeed, all the vertices
colored by the same paint should lie in one $q$-plane of $\Zt^l$.
Otherwise there exist $q+1$ linearly independent vectors which
form a $q$-simplex of $\Ur_l$ colored by the same paint. That
contradicts $q$-regularity of a coloring. Therefore there are at
most $2^q-1$ vertices (which are nonzero vectors) of the same
color. Since there are $2^l-1$ vertices in all, the first
assertion of a theorem follows immediately.

Now we turn to the second statement of the theorem. Let $l =
q\cdot p$. The vector space $\Zt^l$ can be considered as the
$p$-dimensional vector space over the field $\mathbb{F}_{2^q}$. We
color two vertices of $\Ur_l$ by the same paint iff they lie in
the same one-dimensional subspace over $\mathbb{F}_{2^q}$ (thus
representing the same point of the corresponding projective
space). Then any set of vertices colored by a given paint is a
one-dimensional subspace (except zero) of $\mathbb{F}_{2^q}^p$ and
consequently is a $q$-dimensional subspace of $\Zt^l$. This set
does not contain $q+1$ linearly independent vectors. Therefore the
constructed coloring is $q$-regular. There are exactly $2^q-1$
vertices colored by any paint, therefore we have used
$\frac{2^l-1}{2^q-1}$ paints in all. First statement of the
theorem shows that there is no better coloring of $\Ur^l$.
\end{proof}

\textit{Remark.} To prove the second part we have used a fact that
an $l-1$-dimensional projective space over $\Zt$ can be subdivided
into nonintersecting $q-1$-dimensional subspaces. This can be
found in \cite{Hir} in more general situation.

\end{document}